\documentclass[12pt]{amsart}
\usepackage{graphicx, color}
\usepackage{amscd}
\usepackage{amsmath}
\usepackage{amsfonts}
\usepackage{amssymb}
\usepackage{mathrsfs}
\newcommand{\e }{\varepsilon }

\renewcommand{\O }{\Omega}

\renewcommand{\d }{\delta}
\renewcommand{\l }{\lambda }
\newcommand{\io}{\int\limits_\Omega}
\newtheorem{theorem}{Theorem}
\newtheorem{definition}[theorem]{Definition}
\newtheorem{lemma}[theorem]{Lemma}
\newtheorem{proposition}[theorem]{Proposition}
\newtheorem{remark}[theorem]{Remark}
\newtheorem{corollary}[theorem]{Corollary}

\numberwithin{equation}{section}
\newcommand{\RR}{\mathbb R}

\renewcommand{\le}{\leqslant}
\renewcommand{\leq}{\leqslant}
\renewcommand{\ge}{\geqslant}
\renewcommand{\geq}{\geqslant}

\renewcommand{\H}{\mathbb H}

\renewcommand{\epsilon}{\varepsilon}

\newcommand{\Jnt}{\int\limits}
\def\Int{\displaystyle\int}
\begin{document}


\title[]{A nonlocal concave-convex problem with nonlocal mixed boundary data}

\thanks{}


\author[B. Abdellaoui]{Boumediene Abdellaoui}

\address[Boumediene Abdellaoui] {Laboratoire d'Analyse Nonlin\'{e}aire et Math\'{e}matiques Appliqu\'{e}es.
D\'{e}partement de Math\'{e}matiques, Universit\'{e} Abou Bakr Belka\"{i}d, Tlemcen,
Tlemcen 13000, Algeria.}
\email{\tt boumediene.abdellaoui@inv.uam.es}

\author[A. Dieb]{Abdelrazek DIEB}
\address [Abdelrazek Dieb]{Laboratoire d'Analyse Nonlin\'{e}aire et Math\'{e}matiques Appliqu\'{e}es.
Universit\'{e} Abou Bakr Belka\"{i}d, Tlemcen, Tlemcen 13000, Algeria. \\
                    And \\
D\'{e}partement de Math\'{e}matiques, Universit\'{e} Ibn Khaldoun, Tiaret,
Tiaret 14000, Algeria.}
\email{\tt dieb\_d@yahoo.fr}
\author[E. Valdinoci]{Enrico Valdinoci}
\address [Enrico Valdinoci]{University of Melbourne,
School of Mathematics and Statistics,
Richard Berry Building,
Parkville VIC 3010,
Australia,
University of Western Australia,
School of Mathematics and Statistics,
35 Stirling Highway,
Crawley, Perth WA 6009, Weierstra{\ss}-Institut f\"ur Angewandte Analysis und Stochastik,
Mohrenstra{\ss}e 39, 10117 Berlin, Germany,
Dipartimento di Matematica, Universit\`a degli studi di Milano,
Via Saldini 50, 20133 Milan, Italy, and
Istituto di Matematica Applicata e Tecnologie Informatiche,
Consiglio Nazionale delle Ricerche, Via Ferrata 1, 27100 Pavia, Italy.}
\email{\tt enrico@mat.uniroma3.it}

\thanks{ The first author is partially supported by project MTM2013-40846-P, MINECO, Spain. }
.\\


\begin{abstract} The aim of this paper is to study the following problem
$$P_{\lambda} \, \equiv
\left\{
\begin{array}{rcll}
(-\Delta)^s u &= & \lambda u^{q}+u^{p} & {\text{ in }}\O,\\
   u & > & 0 &{\text{ in }} \O, \\
   \mathcal{B}_{s}u &= & 0 & {\text{ in }}\RR^{N}\backslash \O, \\
\end{array}\right.
$$
with $0<q<1<p$, $N>2s$, $\l> 0$, $\Omega \subset \RR^{N} $  is a smooth bounded domain,
$$(-\Delta)^su(x)=a_{N,s}\;
P.V.\Jnt_{\RR^N}\frac{u(x)-u(y)}{|x-y|^{N+2s}}\,dy,$$
$a_{N,s}$ is a normalizing constant, and $\mathcal{B}_{s}u=u\chi_{\Sigma_{1}}+\mathcal{N}_{s}u\chi_{\Sigma_{2}}.$
Here, $\Sigma_{1}$ and $\Sigma_{2}$ are open sets in $\RR^{N}\backslash \Omega$ such that $\Sigma_{1} \cap \Sigma_{2}=\emptyset$ and
$\overline{\Sigma}_{1}\cup \overline{\Sigma}_{2}= \RR^{N}\backslash
\Omega.$

In this setting, $\mathcal{N}_{s}u$ can be seen as a Neumann condition of nonlocal type
that is compatible with the probabilistic interpretation
of the fractional Laplacian, as introduced in \cite{drv},
and $\mathcal{B}_{s}u$ is a mixed Dirichlet-Neumann exterior datum.
The main purpose of this work is to prove existence, nonexistence and multiplicity of positive
energy solutions to problem ($P_{\lambda}$) for suitable ranges of $\l$ and $p$ and to understand the interaction
between the concave-convex nonlinearity and the Dirichlet-Neumann data.
\end{abstract}

\maketitle

\tableofcontents

\section{Introduction}\label{sec:intro}
  In \cite{drv}, the authors introduced a new nonlocal Neumann condition, which is compatible with the probabilistic
interpretation of the nonlocal setting related to some L\'{e}vy process in $\RR^N$.
Motivated by this, we aim in this work to study a semilinear nonlocal elliptic  problem
with mixed Dirichlet-Neumann data. More precisely, we study existence and multiplicity
of positive solutions to the following
problem
$$
P_{\l} \quad \equiv
\left\{
\begin{array}{rcll}
(-\Delta)^s u & = & \lambda u^{q}+u^{p} & {\text{ in }}\O,\\
   u & > & 0 &{\text{ in }} \Omega, \\
   \mathcal{B}_{s}u & = & 0 &{\text{ in }} \RR^{N}\backslash \O,
\end{array}\right.
$$
with $0<q<1<p$, $N>2s$, $\lambda> 0$.

In our setting, $\Omega \subset \RR^{N} $  is a smooth bounded domain
and  $(-\Delta)^s$ is the fractional Laplacian operator,
defined as
\begin{equation}\label{frac lap}
(-\Delta)^su(x)=a_{N,s}\;
P.V.\Jnt_{\RR^N}\frac{u(x)-u(y)}{|x-y|^{N+2s}}\,dy.
\end{equation}
 See e.g. \cite{ldk}, \cite{lenoperal}, \cite{valpal} and the references therein for more information about this operator.
In this framework $a_{N,s}>0$ is a suitable normalization constant and the exterior condition
\begin{equation}\label{CUST}
\mathcal{B}_{s}u=u\chi_{\Sigma_{1}}+\mathcal{N}_{s}u\chi_{\Sigma_{2}},
\end{equation}
can be seen as a nonlocal version of the classical Dirichlet-Neumann mixed boundary condition.
As a matter of fact, here $\mathcal{N}_{s}$ is the non-local normal derivative introduced in \cite{drv}, given by

\begin{equation}\label{Nmn def}
\mathcal{N}_{s}u(x)=a_{N,s}\Int_{\Omega}\frac{u(x)-u(y)}{|x-y|^{N+2s}}dy, \quad x \in \RR^{N}\backslash \overline{\Omega}.
\end{equation}
Also, $\Sigma_{1}$ and $\Sigma_{2}$ are open sets in $\RR^{N}\backslash\Omega$ such that $\Sigma_{1} \cap \Sigma_{2}=\emptyset$ and
$\overline{\Sigma}_{1}\cup \overline{\Sigma}_{2}= \RR^{N}\backslash\Omega.$ As customary,
in \eqref{CUST} we denoted by $\chi_{A}$ the characteristic function of a set $A$.

We observe that, differently from the case of homogeneous Dirichlet conditions,
the case of Neumann and mixed boundary conditions has not been much
investigated in the fractional setting. This is due to the fact that
the classical Neumann condition combines good geometrical properties
(e.g. the normal derivative of the function vanishes, allowing
symmetry and blow-up arguments) and analytic properties, while in the
nonlocal case the consequences of \eqref{Nmn def} are much
less intuitive and harder to deal with. 
This is indeed probably the first article devoted to
the analysis of a nonlinear and nonlocal problem
with mixed exterior data that involve 
the Neumann condition of~\cite{drv}.
We notice that recently a Hopf Lemma has been
proved in \cite{barrios} for such mixed exterior conditions.\medskip

Using an integration by parts formula stated in \cite{drv}, one sees that problem ($P_{\lambda}$) can be set
in a variational setting, since the requested solutions can be seen as critical points of the functional
\begin{equation}\label{JFU}
J_{\l}(u)= \frac{1}{2}\int\int_{\mathcal{D}_{\Omega}}\frac{|u(x)-u(y)|^{2}}{|x-y|^{N+2s}}\,dx\,dy-\frac{\l}{q+1}\|u_+\|^{q+1}_{q+1}-\frac{1}{p+1}\|u_+\|^{p+1}_{p+1},
\end{equation}
where
\begin{eqnarray*}&&
\mathcal{D}_{\Omega}= (\RR^N \times \RR^N )
\setminus( \O^c\times \O^c),\\
&& \|v\|^r_r=\Int_{\O}|v|^r\,dx\quad{\mbox{ and }}\quad
u_+=\max(u,0).\end{eqnarray*}
Such problem, in the local case of the classical Laplacian, was extensively
studied in the literature, especially after the seminal work of Ambrosetti, Brezis and Cerami \cite{ABC}.
Similar problems with a Dirichlet-Neumann datum were studied, for the subcritical case,
in \cite{colopral} and, in the critical case, in \cite{grpc}.

In the nonlocal framework, (that is, when~$s\in(0,1)$), with Dirichlet data, the problem was dealt with in \cite{barrios0} for the
subcritical case and in \cite{bcss}, \cite{manu},
\cite{dmv} and~\cite{BUM} for the critical case. See also \cite{sv}, \cite{sv-DCDS}.

In \cite{bcss} and \cite{dmv}, the authors use an extension method, introduced in \cite{csext}, which allows them to
reduce the problem to a local one. We stress that, in our case, because of the nonlocal Neumann part, we
cannot use such extension and then we deal with the problem in an appropriate purely nonlocal, and somehow more general,
framework. Moreover, to obtain our multiplicity result, we have to use an additional argument which was classically developed by Alama in
\cite{alm}.

For a series of motivations about nonlocal equations
and fractional operators, see e.g.~\cite{BUC} and the references
therein.
\medskip

Our main results are the following:
\begin{theorem}\label{th1}
Let $0<s<1$, $0<q<1 <p$. Then there exists $\Lambda>0$, such that:
\begin{enumerate}
  \item For all $\lambda \in (0,\Lambda)$, problem ($P_{\lambda}$) has a minimal solution $u_{\lambda}$ such that $J_{\lambda}(u_{\lambda})<0$.
      Moreover, these solutions are ordered,
namely: if $\lambda_{1}<\lambda_{2}$ then $u_{\lambda_{1}}<u_{\lambda_{2}}$.
  \item If $\lambda>\Lambda$, problem ($P_{\lambda}$) has no positive weak solutions.
  \item If $\lambda=\Lambda$, problem ($P_{\lambda}$) has at least one weak positive solution.
\end{enumerate}
\end{theorem}

\begin{theorem}\label{th2}
For all  $0<s<1$, $0<q<1<p<\frac{N+2s}{N-2s}$, $\l \in (0,\Lambda)$,
problem ($P_{\l}$) has a second solution $v_{\l}>u_{\l}$.
\end{theorem}

The paper is organized as follows: In Section 2, we introduce the functional setting to deal with problem ($P_{\lambda}$),
as well as the notion of solution we will work with and some auxiliary results. Section 3 is devoted to prove the existence
of minimal and extremal solutions. Finally in Section 4 we prove the existence of a second solution using Alama's argument.

\section{Preliminaries and functional setting.}

 We introduce in  this section a natural functional framework for our problem and we give some related properties and some useful
embedding results needed when we deal with problem ($P_{\l}$).
According to the definition of the fractional Laplacian, see
\cite{valpal}, \cite{sv}, and the integration by parts formula, see
\cite{drv} , it is natural to introduce the following spaces. We
denote by $H^s(\RR^N)$ the classical fractional Sobolev space,
\begin{equation}\label{}
H^s(\RR^N)=\left\{u \in
L^2(\RR^N):\frac{|u(x)-u(y)|}{|x-y|^{\frac{N}{2}+s}} \in
L^2(\RR^N\times \RR^N)\right\},
\end{equation}
endowed with the norm
\begin{equation}\label{}
    \|u\|^2_{H^s(\RR^N)}=
\|u\|_{L^2(\RR^N)}^2 + \int\int_{\RR^N\times \RR^N}\frac{|u(x)-u(y)|^{2}}{|x-y|^{N+2s}}\,dx\,dy.
\end{equation}
It is clear that  $H^s(\RR^N)$ is a Hilbert space.

We recall now the classical Sobolev inequality that the proof can be found in \cite{valpal}. See also \cite{Ponce} for an elementary proof.
\begin{proposition}\label{PSO}
 Let $s\in (0,1)$ with $N>2s$. There exists a positive constant $S=S(N,s)$ such that,
for any function $u \in H^{s}(\RR^N)$, we have
\begin{equation}\label{sobo}
    S \|u\|_{L^{2^{*}_{s}}(\RR^N)}^2 \leq \Int\Int_{\RR^{N}\times\RR^{N}}\frac{|u(x)-u(y)|^{2}}{|x-y|^{N+2s}}\,dx\,dy,
  \end{equation}
where $2^{*}_{s}=\frac{2N}{N-2s}$.
\end{proposition}
\begin{definition}\label{def1}
Let $\O$ be a bounded domain of $\RR^N$. For $0<s<1$, we define the space
$$\H^{s}(\Omega,\Sigma_1)=\left\{ u \in H^s(\RR^N):\, u=0 \text{ in } \Sigma_{1}\right\}.$$
It is clear that $\H^{s}(\Omega,\Sigma_1)$ is a Hilbert space endowed with the norm induced by $H^s(\RR^N)$.
\end{definition}
For $u\in \H^{s}(\Omega,\Sigma_1)$, we set

$$\|u\|^{2}=a_{N,s}\int\int_{\mathcal{D}_{\Omega}}\frac{|u(x)-u(y)|^{2}}{|x-y|^{N+2s}}\,dx\,dy.$$
The properties
of this norm are described by the following result.
\begin{proposition}\label{norm eqv}
The norm $\|\,.\, \|$ in $\H^{s}(\Omega,\Sigma_1)$ is equivalent to the one induced by $H^s(\RR^N)$, and then $(\H^{s}(\O,\Sigma_1), \langle\, ,\,\rangle)$ is a Hilbert space with scalar product given by
$$\langle u,v\rangle=a_{N,s}\int\int\limits_{\mathcal{D}_{\Omega}}\frac{(u(x)-u(y))(v(x)-v(y))}{|x-y|^{n+2s}}\,dx\,dy.$$
\end{proposition}
\begin{proof}
For $u\in \H^{s}(\Omega,\Sigma_1)$, we set
$$
\|u\|^2_1=\|u\|^2_{L^2(\O)}+a_{N,s}\int\int_{\mathcal{D}_{\Omega}}\frac{|u(x)-u(y)|^{2}}{|x-y|^{N+2s}}\,dx\,dy.
$$
It is clear that $\H^{s}(\Omega,\Sigma_1)$ is a Hilbert space with the associated scalar product given by
$$\langle u,v\rangle_1=a_{N,s}\int\int\limits_{\mathcal{D}_{\Omega}}\frac{(u(x)-u(y))(v(x)-v(y))}{|x-y|^{n+2s}}\,dx\,dy+\io uv\,dx.$$
Notice that the completeness of $\H^{s}(\Omega,\Sigma_1)$ can be proved using exactly the same argument as in the proof of of Proposition 3.1 in \cite{drv}.

Now, setting
$$
\l_1(\O)=\inf_{\{\phi\in \H^{s}(\Omega,\Sigma_1), \phi\neq 0\}}\dfrac{a_{N,s}\displaystyle\int\int_{\mathcal{D}_{\Omega}}\frac{|\phi(x)-\phi(y)|^{2}}{|x-y|^{N+2s}}\,dx\,dy}{\displaystyle\io \phi^2\,dx},
$$
then the authors in \cite{barrios} proved that $\l_1(\O)>0$. As a consequence, the previous scalar product can be reduced to the following one
$$\langle u,v\rangle=a_{N,s}\int\int\limits_{\mathcal{D}_{\Omega}}\frac{(u(x)-u(y))(v(x)-v(y))}{|x-y|^{n+2s}}\,dx\,dy.$$
Hence we can endowed $\H^{s}(\Omega,\Sigma_1)$ with the Gagliardo norm
$$\|u\|^{2}=a_{N,s}\int\int_{\mathcal{D}_{\Omega}}\frac{|u(x)-u(y)|^{2}}{|x-y|^{N+2s}}\,dx\,dy.$$
Now, the norm $\|\,.\, \|$ in $\H^{s}(\Omega,\Sigma_1)$ is 
bounded by
the one induced by $H^s(\RR^N)$, and so,
using the Open Mapping Theorem, it holds that 
the norm $\|\,.\, \|$ in $\H^{s}(\Omega,\Sigma_1)$ is in fact
equivalent to the one induced by $H^s(\RR^N)$. Hence the result follows.
\end{proof}
The following result justifies our choice of $\|\,.\,\|$.
\begin{proposition}\label{BYP}
 Let $s \in (0,1)$, for all $u,\,v \in \H^{s}(\Omega,\Sigma_1)$ we have,
     $$ \Int_{\O}v(-\Delta)^su\,dx=
\frac{a_{N,s}}{2}\int\int\limits_{\mathcal{D}_{\Omega}}
     \frac{(u(x)-u(y))(v(x)-v(y))}{|x-y|^{N+2s}}\,dx\,dy -\Int_{\Sigma_2}v\mathcal{N}_su\,dx.$$
\end{proposition}

 The proof of this result is a direct application of the integration by parts
formula, see Lemma 3.3 in \cite{drv}.

In the rest of the paper, for the simplicity of typing, we shall denote
the functional space introduced in Definition \ref{def1} by $\H^{s}$ and we shall normalize the constant $a_{N,s}$ to be equal to $2$.

Now we give a Sobolev-type result for functions in $\H^s$.

\begin{corollary}\label{COROSO}
Suppose that $s\in (0,1)$ and $N>2s$. There exists a positive constant $C=C(N,s,\O,\Sigma_2)$ such that, for any function $u \in \H^{s}$,
$$
  \begin{array}{ccc}
    \|u\|_{L^{r}(\O)}^2& \leq &
    C \|u\|^2,
  \end{array}
$$
for all $1\le r\leq 2^{*}_{s}$.
\end{corollary}
\begin{proof}
Since $u\in \H^{s}\subset H^s(\RR^N)$, then using the Sobolev inequality in \eqref{PSO}, it holds that
$$
S\|u\|_{L^{2^{*}_{s}}(\O)}^2\le  S \|u\|_{L^{2^{*}_{s}}(\RR^N)}^2 \leq \Int\Int_{\RR^{N}\times\RR^{N}}\frac{|u(x)-u(y)|^{2}}{|x-y|^{N+2s}}\,dx\,dy.
$$
Now, the result follows using H\"older inequality and Proposition \ref{norm eqv}.
\end{proof}
Consider now the standard truncation functions given by $$
T_k(u)= \max\big\{-k,\; \min\{k, u\}\big\}$$ and $G_k(u) = u - T_k(u)$.
 In this setting, the following are some useful properties of $\H^s$-functions which are needed to get some regularity
results for some elliptic problems in $\H^s$ (see also Theorem \ref{threg} below).

\begin{proposition}\label{g_k t_k}
Let $u$ be a function in $\H^{s}$, then
\begin{enumerate}
    \item  if $\Phi \in Lip(\RR)$ is such that $\Phi(0)=0$, then $\Phi(u) \in \H^{s}$. In particular for any $k>0$, $T_{k}(u),\, G_{k}(u)\in
\H^{s}.$
    \item For any $k\geq 0$
        $$ \|G_{k}(u)\|^{2}\leq \int_{\O}G_{k}(u)(-\Delta)^{s}u\,dx+\Int_{\Sigma_2} G_{k}(u)\mathcal{N}_su\,dx.$$
    \item For any $k\geq 0$
        $$ \|T_{k}(u)\|^{2}\leq \int_{\O}T_{k}(u)(-\Delta)^{s}u\,dx+\Int_{\Sigma_2}T_{k}(u)\mathcal{N}_su\,dx.$$

\end{enumerate}
\end{proposition}
\begin{proof} The claim in (1) follows from the setting of the norm given in Definition \ref{def1}.
As for (2) and (3), we claim that, for any $a$, $b\ge0$ and any $x\in\RR^N$,
\begin{equation}\label{INE}
a\,\big(G_{k}(u)(-\Delta)^{s} T_k(u)\big)(x)+
b\,\big(G_{k}(u)\mathcal{N}_s T_k(u)\big)(x)\ge0.
\end{equation}
To check this, we can take $x\in \{ G_k(u)\ne0\}$,
otherwise \eqref{INE} is obvious. Then, if $x\in \{ G_k(u)>0\}$
we have that $T_k(u)(x)=k$, which is the maximum value that $T_k(u)$
attains, and therefore $(-\Delta)^{s} T_k(u)(x)\ge0$
and $\mathcal{N}_s T_k(u)(x)\ge0$. 

Conversely, if $x\in \{ G_k(u)<0\}$
we have that $T_k(u)(x)=-k$, which is the minimum value that $T_k(u)$
attains, and therefore $(-\Delta)^{s} T_k(u)(x)\le0$
and $\mathcal{N}_s T_k(u)(x)\le0$. By combining these observations,
we obtain \eqref{INE}. {F}rom \eqref{INE} and Proposition \ref{BYP} it follows that
\begin{equation}\label{INE2}
\begin{split} &
\Int_{\O} T_k(u)(-\Delta)^s G_k(u)\,dx
+\Int_{\Sigma_2} T_k(u) \mathcal{N}_s G_k(u)\,dx
\\&\qquad\quad=
\Int_{\O} G_k(u)(-\Delta)^s T_k(u)\,dx
+\Int_{\Sigma_2} G_k(u) \mathcal{N}_s T_k(u)\,dx
\ge0.
\end{split}\end{equation}
Using the normalization condition and by Propositions \ref{BYP}, we reach that
\begin{equation}\label{SY1}
\begin{split}
\|G_k(u)\|^2\,&=
\int\int\limits_{\mathcal{D}_{\Omega}}
     \frac{(G_k(u)(x)-G_k(u)(y))^2}{|x-y|^{N+2s}}\,dx\,dy
\\ &= \Int_{\O} G_k(u)(-\Delta)^sG_k(u)\,dx+
\Int_{\Sigma_2}G_k(u)\mathcal{N}_sG_k(u)\,dx
\\ &= \Int_{\O} G_k(u)(-\Delta)^s\big(u-T_k(u)\big)\,dx+
\Int_{\Sigma_2}G_k(u)\mathcal{N}_s\big(u-T_k(u)\big)\,dx.
\end{split}
\end{equation}
In a similar way,
\begin{equation}\label{SY2}
\begin{split}
\|T_k(u)\|^2=
\Int_{\O} T_k(u)(-\Delta)^s\big(u-G_k(u)\big)\,dx+
\Int_{\Sigma_2}T_k(u)\mathcal{N}_s\big(u-G_k(u\big)\,dx
.\end{split}
\end{equation}

Then, the claim in (2) follows from \eqref{SY1} and \eqref{INE2}, while the claim in (3) follows from
\eqref{SY2} and \eqref{INE2}.
\end{proof}

 Let us now consider the  following problem,
\begin{equation}\label{pb}
\left\{
\begin{array}{rcll}
(-\Delta)^s u & = & f & {\text{ in }}\Omega,\\
   \mathcal{B}_{s}u & = & 0 &{\text{ in }} \RR^{N}\backslash \O,
\end{array}\right.
\end{equation}
where $\O $ is a bounded regular domain of $\RR^{N}$, $ N> 2s$,
$\H^{-s}$ is the dual space of $\H^{s}$ and $f\in \H^{-s}$.
\begin{definition}
 We say that $u \in \H^{s}$ is an energy solution to (\ref{pb}) if
\begin{equation}\label{eng sol}
  \Int\Int_{\mathcal{D}_{\O}}\frac{ \big(u(x)-u(y)\big)
  \big(\varphi(x)-\varphi(y)\big)}{|x-y|^{N+2s}}\,dx\,dy=(f,\varphi)
  \quad \forall \varphi \in \H^{s},
\end{equation}
where $(\, ,\, )$ represent the duality between $\H^{s}$ and
$\H^{-s}.$
\end{definition}

Notice that the existence and uniqueness of energy solutions to
problem (\ref{pb}) follow from the Lax-Milgram Theorem.
Furthermore if $f\geq 0$ then $u \geq 0$. Indeed for $u \in \H^s$, thanks
to Lemma \ref{g_k t_k}, we know that $u_-=\min\{u,0\} \in \H^s$.
Taking $u_-$ as a test function in (\ref{eng sol}) it follows that $u_- =0$.

 A supersolution (respectively, subsolution) is a function that verifies (\ref{eng sol}) with equality replaced by
``$\geq$'' (respectively, ``$\leq$'') for every non-negative test function in $\H^{s}.$
 Using a standard iterative argument we can easily prove the following result.

\begin{lemma}\label{IT6}
Assume that problem (\ref{pb}) has a subsolution $\underline{w}$
and a supersolution $\overline{w}$, verifying
$\underline{w}\leq\overline{w}$. Then there exists a solution $w$
satisfying $\underline{w}\leq w \leq\overline{w}.$
\end{lemma}

 Here we prove some regularity results when $f$ satisfies some
minimal integrability conditions. To prove the boundedness  of the
solution we follow the idea of Stampacchia for second order
elliptic equations with bounded coefficients. The interior H\"{o}lder
regularity is a consequence of continuity properties, see
\cite{drv}, and the regularity results in \cite {sv wek}.
\begin{lemma}\label{lmreg} Let $u$ be a solution to problem (\ref{pb}). If $f \in L^{q}(\O)$, $q > \frac{N}{2s}$, then $u \in L^{\infty}(\O).$
\end{lemma}
\begin{proof}
We follow here a related argument presented in \cite{lenoperal}. See also \cite{sv wek} and \cite{manu}
for related results.
Let $k > 0$ and take $\varphi = G_k(u)$ as a test function in (\ref{eng sol}).
Hence, thanks to Proposition \ref{g_k t_k}, we get
$$ \|G_{k}(u)\|^{2}\leq \int_{A_k}G_{k}(u)f\,dx+\Int_{\Sigma_2}G_{k}(u)\mathcal{N}_su\,dx,$$
where $A_k =\{x \in \O : u > k\}$. Recalling (\ref{pb}), we obtain
$$ \|G_{k}(u)\|^{2}\leq \int_{A_k}G_{k}(u)f\,dx.$$
Applying Corollary \ref{COROSO} in the left hand side and H\"{o}lder inequality in the right hand side, we obtain
$$ S^2\|G_k(u)\|_{L^{2^{*}_{s}}(\O)}^2\leq\|G_{k}(u)\|^{2}\leq \|f\|_{L^{m}(\O)} \|G_k(u)\|_{L^{2^{*}_{s}}(\O)}|A_k|^{1-\frac{1}{2^{*}_{s}}-\frac{1}{m}}$$
we have that,
$$ S^2\|G_k(u)\|_{L^{2^{*}_{s}}(\O)}^2\leq \|f\|_{L^{m}(\O)}|A_k|^{1-\frac{1}{2^{*}_{s}}-\frac{1}{m}}$$
thus,
$$ S^2(h-k)|A_h|^{\frac{1}{2^{*}_{s}}}\leq \|f\|_{L^{m}(\O)}|A_k|^{1-\frac{1}{2^{*}_{s}}-\frac{1}{m}}$$
and then,
$$|A_h|\leq S^{2^{*}_{s}-2} \frac{\|f\|_{L^{m}(\O)}^{2^{*}_{s}}|A_k|^{2^{*}_{s}(1-\frac{1}{2^{*}_{s}}-\frac{1}{m})}}{(h-k)^{2^{*}_{s}}}.$$
Since $m > \frac{N}{2s}$ we have that
$$2^{*}_{s}\,\left(1-\frac{1}{2^{*}_{s}}-\frac{1}{m}\right) > 1.$$
Hence we apply Lemma 14 in \cite{lenoperal} with $\psi(\sigma) = |A_\sigma|$ and the result follows.
\end{proof}


\begin{corollary}
 Let $u$ be an energy solution of (\ref{pb}) and suppose that $f \in L^{\infty}(\O)$.
Then $u \in C^{\gamma}(\overline{\O})$, for some $\gamma \in (0,1)$.
\end{corollary}

\begin{proof}
We claim that $u$ is bounded in $\RR^N$. Then one could apply interior regularity
results for the solutions to $(-\Delta)^s u=0 \in \Omega \,\, \hbox{and}\,\, u=g \,\hbox{ in }\,\,\Omega^{c}$. See e.g. \cite{sv wek} and~\cite{ros srvy}.

To check the claim, recalling Lemma \ref{lmreg}, we have to consider only the case $x\in \overline{\Sigma}_2$. Then, by (\ref{Nmn def})
 \begin{equation*}
        u(x)= c(N,s)^{-1}\Int_{\O}\frac{u(y)}{|x-y|^{N+2s}}\, dy \,\, \hbox{, where}\, \,c(N,s)= \Int_{\O}\frac{1}{|x-y|^{N+2s}}\,dy.
 \end{equation*}
 Hence,
 \begin{equation}\label{clm reg 1}
   |u(x)|\leq \|u\|_{L^{\infty}(\O)} \,\,\hbox{for all }\,\, x\in \overline{\Sigma}_2.
 \end{equation}
 Also, if $\Sigma_2$ is unbounded, using Proposition 3.13 in \cite{drv}, we have
 \begin{equation}\label{clm reg 2}
   \lim_{x\rightarrow \infty,\, x\in \overline{\Sigma}_2}u(x)=\frac{1}{|\O|} \Int_{\O}u(y)\,dy.
 \end{equation}
Then the claim follows from  Lemma \ref{lmreg}, inequalities (\ref{clm reg 1}) and (\ref{clm reg 2}).
\end{proof}
 As a variation of Lemma \ref{lmreg}, we point out that if $f=f(x,u)$ and $f$ has the following growth
 \begin{equation}\label{growth}
 |f(x,s)|\leq c (1+|s|^p) \,\, \hbox{where}\, \, p\leq \frac{N+2s}{N-2s},
\end{equation}
 then, using a Moser iterative scheme, we can prove that:
\begin{theorem}\label{threg}
Let $u$ be an energy solution to problem (\ref{pb}) with $f$ satisfies (\ref{growth}), then $u\in L^\infty (\O)$.
\end{theorem}
 The following is a strong maximum principle for semi-linear
equations, it will be used to separate minimal solution of problem
($P_\l$) for different values of the parameter $\l$, see \cite{mteo}.
\begin{proposition}\label{strg max}
Let $N\geq1$, $0<s<1$ and let $f_1$, $f_2:\RR^N\times\RR \rightarrow \RR$
be two continuous functions. Let $\O$ be a domain in $\RR^N$ and
$v,w \in L^{\infty}(\RR^N)\cap C^{2s+\gamma}(\O)$, for some $\gamma>0$, be such that
 $$
\left\{
\begin{matrix}
  (-\Delta)^{s} v \geq f_1(x,v), &\text{in}& \O, \\
                     &                 \\
  (-\Delta)^{s} w \leq f_2(x,w), &\text{in}& \O, \\
                     &                  \\
         v\geq w                 & \text{in}& \RR^N.
\end{matrix}
\right.
$$
Suppose furthermore that
\begin{equation}\label{groth prmax}
  f_2(x,w(x))\leq f_1(x,w(x)) \text{  for any } x \in \O.
\end{equation}
If there exists a point $x_0 \in \O$ at which $v(x_0)=w(x_0)$, then $v=w$ in the whole $\O$.
\end{proposition}

\begin{proof}
Let $\phi=v-w$ and set
$$Z_{\phi}= \left\{ x\in \O: \phi(x)=0 \right\}.$$
By assumption $x_0 \in Z_{\phi}$. Moreover, thanks to the continuity of $\phi$, we know that
$Z_{\phi}$ is closed.
We claim now that $Z_{\phi}$ is also open. Indeed, let $\bar{x} \in Z_{\phi}$. Clearly $\phi \geq 0$ in $\RR^N$, $\phi(\bar{x})=0$ and
$$(-\Delta)^s \phi(\bar{x})\geq f_1(\bar{x},v(\bar{x}))-f_2(\bar{x},w(\bar{x}))=f_1(\bar{x},w(\bar{x}))-f_2(\bar{x},w(\bar{x}))\geq 0,$$
in view of (\ref{groth prmax}). Accordingly,
$$
\begin{matrix}
  0 & \leq& (-\Delta)^s\phi(\bar{x})&=&\displaystyle\frac{1}{2}\Int_{\RR^N}\frac{2\phi(\bar{x})-\phi(\bar{x}+z)-\phi(\bar{x}-z)}{|z|^{N+2s}}\,dz \\
    & & & =& \displaystyle\frac{1}{2}\Int_{\RR^N}\frac{-\phi(\bar{x}+z)-\phi(\bar{x}-z)}{|z|^{N+2s}}\,dz \leq 0. \quad \quad \quad\\
\end{matrix}
$$
Hence $\phi $ vanishes identically in $B_{\epsilon}(\bar{x})$ and then, for $\epsilon$ small, $B_{\epsilon}(\bar{x}) \subseteq Z_{\phi}$.
That is, we have proved that $Z_{\phi}$ is open, and so, by the connectedness of $\O$, we get that $Z_{\phi}=\O$.
\end{proof}
 Now we establish two important results for our purposes. The first result
is a Picone-type inequality and the second is a Brezis-Kamin comparison
principle for concave nonlinearities.

\begin{theorem}\label{PICONE}
Consider $u,\, v \in \H^{s}$, suppose that $(-\Delta)^{s}u\geq 0$ is
a bounded Radon measure in $\O$, $u\geq0$ and not identically zero,
then,
$$\Int_{\Sigma_2}\frac{|v|^2}{u}\mathcal{N}_su\,dx+\Int_{\O}\frac{|v|^2}{u}(-\Delta)^{s}u\, dx \leq \Int\Int_{\mathcal{D}_{\O}}\frac{\big(v(x)-v(y)\big)^{2}}{|x-y|^{N+2s}}\,dx\,dy. $$
\end{theorem}

The proof of this result is based on a punctual inequality and follows in the same
way as in \cite{lenoperal}.
As a consequence, we have the next comparison principle that extends to the fractional framework the classical
one obtained by Brezis and Kamin, see \cite{bk}.

\begin{lemma}\label{lmbk}
Let $f(x,\sigma)$ be a Carath\'{e}odory function such that
$\frac{f(x,\sigma)}{\sigma}$ is decreasing in $\sigma$, uniformly
with respect to $x \in \O$. Suppose that $u, v \in \H^{s}$, with $0<s<1$,
are such that
$$
\left\{
\begin{matrix}
  (-\Delta)^{s} u \geq f(x,u), & u>0 \quad \text{in} \ \O, \\
                     &                 \\
  (-\Delta)^{s} v \leq f(x,v), & v>0 \quad \text{in} \ \O.
\end{matrix}
\right.
$$
Then $u\geq v$ in $\O$.
\end{lemma}
The proof of this result is a slight modification of the proof of Theorem 20 in
\cite{lenoperal}.
 Finally, we will use the following compactness lemma to get strong
convergence in the space $\H^s$.
\begin{lemma}\label{lmcp}
 Let $\{v_n\}_n$ be a sequence of non-negative functions such that $\{v_n\}_n$
 is bounded in $\H^s$, $v_n \rightharpoonup v$ in $\H^s$ and $v_n \leq v$. Assume that $(-\Delta)^s v_n \geq0$ then,
   $v_n \rightarrow v$ strongly in $\H^s$.
\end{lemma}
\begin{proof}

 Since $v_n \leq v$, then using the fact that $(-\Delta)^s v_n \geq 0$, it follows that
$$\begin{array}{ccc}
  \Int_{\O}(-\Delta)^s v_n (v-v_n)\,dx&\geq & 0. \\
\end{array}
$$
Hence
$$\begin{array}{ccc}
  \Int_{\O}(-\Delta)^s v_n v\,dx&\geq & \Int_{\O}(-\Delta)^s v_n v_n\,dx. \\
\end{array}$$
Now, using Young's inequality, we obtain that
$$
\begin{array}{ccc}
  \Int\Int_{\mathcal{D}_{\O}}\frac{
\big(v_n(x)-v_n(y)\big)^{2}}{|x-y|^{N+2s}}\,dx\,dy & \leq &
\Int\Int_{\mathcal{D}_{\O}}\frac{\big(v(x)-v(y)\big)^{2}}{|x-y|^{N+2s}}\,dx\,dy. \\
\end{array}
$$
Thus
$$\begin{array}{ccc}
  \limsup\limits_{n\rightarrow \infty} \|v_n\| & \leq & \|v\|. \\
\end{array}
$$
Since
\begin{eqnarray*}
\limsup_{n\rightarrow \infty} \|v_n-v\|^2&=&
\limsup_{n\rightarrow \infty} (\|v_n\|^2+\|v\|^2-2\langle v_n,v\rangle)\\
&\leq& 2\|v\|^2-2
\limsup_{n\rightarrow \infty} \langle v_n,v\rangle,
\end{eqnarray*}
taking into consideration that $v_n \rightharpoonup v$ in $\H^s$, we get
$$\limsup_{n\rightarrow \infty} \|v_n-v\|^2=0.$$
As a consequence, $v_n\rightarrow v$ strongly in $\H^s.$
\end{proof}

\section{Proof of Theorem \ref{th1}}

 In this section we prove Theorem \ref{th1}. We split the proof into several auxiliary Lemmas.
 Let us begin by proving an existence result.
\begin{lemma}\label{lm00}
Assume that $0<q<1<p$, then problem ($P_{\lambda}$) has a nontrivial bounded solution at least for $\l>0$ small.
\end{lemma}
\begin{proof}
The main idea is to show that for $\l$ small, the problem ($P_{\lambda}$) has a comparable bounded sub and supersolution.
Let $\mathcal{V}$ be the unique positive solution to the problem
\begin{equation*}
\left\{
\begin{array}{rcll}
(-\Delta)^s \mathcal{V} & = & 1 & {\mbox{ in }}\O,\\
   \mathcal{V} & > & 0 &{\mbox{ in }} \O, \\
   \mathcal{B}_{s}\mathcal{V} & = & 0 &{\mbox{ in }} \RR^{N}\backslash \O.
\end{array}\right.
\end{equation*}
Notice that the existence of $\mathcal{V}$ follows by
using the Lax-Milgram theorem in the space $\H^s$, however the positivity of $\mathcal{V}$ follows form \cite{barrios}. It is
clear that $\mathcal{V}\in \mathcal{C}^\alpha(\bar{\O})$ for some $\alpha<1$. Let $C=\|
\mathcal{V}\|_\infty$, it is not difficult to show the existence of $\l^*>0$ such that for all $\l<\l^*$, the inequality
$$
M\ge \l M^q C^q+M^pC^p,
$$
has a solution $M>0$. Fix $\l, M$ as above and define $v_1=M\mathcal{V}$, then $v_1$ solves
\begin{equation}\label{exit0}
\left\{
\begin{array}{rcll}
(-\Delta)^s v_1 & = & M\ge \l v^q_1+v^p_1 & {\mbox{ in }}\O,\\
   v_1 & > & 0 &{\mbox{ in }} \O, \\
   \mathcal{B}_{s} v_1 & = & 0 &{\mbox{ in }} \RR^{N}\backslash \O.
\end{array}\right.
\end{equation}
Thus $v_1$ is a supersolution to problem ($P_{\lambda}$).

We consider now the following problem
\begin{equation}\label{pbcnv}
\left\{
\begin{array}{rcll}
(-\Delta)^s z & = & z^{q}& {\mbox{ in }}\O,\\
   z & > & 0 &{\mbox{ in }} \O, \\
   \mathcal{B}_{s}z & = & 0 &{\mbox{ in }} \RR^{N}\backslash \O.
\end{array}\right.
\end{equation}
Since $q\in (0,1)$, then setting
$$M=\min\left\{ \frac{1}{2}\|w\|^2-\frac{\l}{q+1}\Int_{\O}w_+^{q+1}dx,\, \, w\in \H^s \right\},$$
it follows that $M$ is achieved by a minimizer $z$. It is clear 
that $z\ge0$,
then by Proposition \ref{strg max} and Lemma \ref{lmbk}, it follows that $z> 0$ and it is unique.
In particular, $z$ is the solution to problem (\ref{pbcnv}).
By Theorem \ref{threg}, it holds that $z\in L^{\infty}(\O)$.

Now setting $z_\l=\l^{\frac{1}{1-q}}z$, then $z_\l$ is a solution to
\begin{equation}\label{l11}
\left\{
\begin{array}{rcll}
(-\Delta)^s z_\l& = & \lambda z_\l^{q} & {\mbox{ in }}\O,\\
   \mathcal{B}_{s} z_\l & = & 0 &{\mbox{ in }} \RR^{N}\backslash \O.
\end{array}\right.
\end{equation}
By the comparison result in Lemma \ref{lmbk}, it holds that $z_\l\leq v_1$. 
It is clear that $z_\l$ is
a subsolution to problem ($P_{\lambda}$). Hence a monotonicity argument allows us
to get the existence of a solution $u_\l$ to problem ($P_{\lambda}$) with $z_\l\le u_\l\le v_1$.
\end{proof}

\begin{lemma}\label{lm1}
Let $\Lambda$ be defined by
$$\Lambda = \sup\left\{ \lambda > 0: \, \, \text{problem ($P_{\lambda}$) has a solution }  \right\}.$$
Then $0<\Lambda< \infty$.
\end{lemma}

\begin{proof}
By Lemma \ref{lm00}, we reach that $\Lambda>0$.

We show now that $\Lambda<\infty$. Let $\lambda$ be such that problem ($P_{\lambda}$)
has a solution $\bar{u}_{\lambda}$. By the comparison principle in Lemma \ref{lmbk},
we get $z_\l\leq\bar{u}_{\lambda}$ where $z_\l$ is the unique positive solution
to problem \eqref{l11}. Let $\phi\in \H^s$, then using Picone's inequality we obtain that
$$
\begin{matrix}
    \Int\Int_{\mathcal{D}_{\O}}\frac{\big(\phi(x)-\phi(y)\big)^{2}}{|x-y|^{N+2s}}\,dx\,dy &\geq&
    \Int_{\O}\frac{\phi^2}{\bar{u}_{\lambda}}(-\Delta)^{s}\bar{u}_{\lambda}\,dx\\
     &\geq& \Int_{\O}\phi^2 (\l \bar{u}^{q-1}_{\l}+\bar{u}^{p-1}_{\l})\, dx \\
         & \geq & \Int_{\O}z_\l^{p-1} \phi^{2}\,dx  \\
     & \geq &  \lambda^{\frac{p-1}{1-q}}\Int_{\O}z^{p-1}\phi^{2}\,dx.\\
\end{matrix}
$$
Hence
\begin{equation}\label{}
  \lambda^{\frac{p-1}{1-q}}\leq \inf_{\phi\in \H^s }\frac{\Int\Int_{\mathcal{D}_{\O}}\frac{\big(\phi(x)-\phi(y)\big)^{2}}
  {|x-y|^{N+2s}}\,dx\,dy}{\Int_{\O}z^{p-1}\phi^{2}\,dx}=\Lambda^{*}.
\end{equation}
Consequently, $\Lambda \leq \left(\Lambda^{*}\right)^{\frac{1-q}{p-1}}< \infty$. This gives point (2) in Theorem \ref{th1}.
\end{proof}

We show now that for all $0<\l<\Lambda$, problem ($P_\l$) has a solution. This will be a consequence of the following lemma.
\begin{lemma}
Let
\begin{equation}\label{INTERVAL}
S=\left\{ \lambda > 0: \, \, \text{problem ($P_{\lambda}$)
has a solution}\right\}.
\end{equation}
Then $S$ is an interval.
\end{lemma}
\begin{proof}
Notice that $S\neq \emptyset$, thanks to Lemma~\ref{lm00}. Let $\l_1\in S$ be fixed, we have just to prove
that for all $0<\l_2<\l_1$, problem $(P_{\lambda_2})$ has a non trivial solution.

Since $\l_1\in S$, then we get the existence of $u_1\in \H^{s}$ such that $u_1$ solves $(P_{\lambda_1})$.
It is clear that $u_1$ is a supersolution to problem $(P_{\lambda_2})$. Recall that $z$ is the unique solution
to problem \eqref{pbcnv}. Setting $z_2=\l_2^{\frac{1}{1-q}}z$, then $z_2$ solves
\begin{equation*}\label{zll}
\left\{
\begin{array}{rcll}
(-\Delta)^s z_2 & = & \lambda_2 z_2^{q} & {\mbox{ in }}\O,\\
   \mathcal{B}_{s}z_2 & = & 0 &{\mbox{ in }} \RR^{N}\backslash \O.
\end{array}\right.
\end{equation*}
By the comparison principle in Lemma \ref{lmbk}, it holds that $z_2\le u_1$.

Since $z_2$ is a subsolution to problem $(P_{\lambda_2})$, then using a monotonicity argument we
get the existence of $u_2\in \H^{s}$ such that $z_2\le u_2\le u_1$ and $u_2$
solves problem  $(P_{\lambda_2})$. Thus $\l_2\in S$ and the result follows.
\end{proof}

 We now prove that ($P_\l$) possesses a minimal solution and we give some energy properties of such solutions.
\begin{lemma} \label{POINT}
For all $0<\lambda<\Lambda$, problem ($P_{\lambda}$) has a minimal solution $u_{\lambda}$ such that
$J_{\l}(u_{\l})<0$. Moreover the family ${u_{\lambda}}$ of minimal
solutions is increasing with respect to $\lambda$.
\end{lemma}
\begin{proof}
Suppose that ($P_{\lambda}$) has a solution $v_\l$ for a given
$\lambda\in S$. Define the sequence $v_{n}$ by
$v_{0}=z_\l$,
\begin{equation}\label{aprox pb}
\left\{
\begin{array}{rcll}
(-\Delta)^s v_{n} & = & \lambda v_{n-1}^{q}+v_{n-1}^{p} & {\text{ in }}\O,\\
   v_{n} & \geq & 0 &{\text{ in }} \Omega, \\
   \mathcal{B}_{s}v_{n} & =& 0 &{\text{ in }} \RR^{N}\backslash \O,\\
\end{array}\right.
\end{equation}
where $z_\l$ is the unique solution to problem \eqref{l11}. By the comparison result in Lemma
\ref{lmbk}, we have that $\bar{z}\leq...\leq v_{n-1}\leq v_{n}\leq v_{\l}$ and then,
by Proposition \ref{strg max}, it follows that $z_\l<v_{n}< v_{\l}$.

So, using $v_n$ as a test function in (\ref{aprox pb}), we get $\|v_n\|\leq \|v_\l\|$.
Hence there exists $u_\l \in \H^s$ such that $v_n \rightharpoonup u_\l$.
Accordingly, since $(-\Delta)^sv_n \geq 0$, using Lemma \ref{lmcp}, we
conclude that $v_{n}\rightarrow u_{\l}$ strongly in $\H^s$ and $u_{\l}\leq v_{\l}$.
This shows that $u_{\l}$ is a minimal solution.

Then, by Lemma \ref{lmbk} and Proposition \ref{strg max},
we obtain the monotonicity of the family $\left\{u_{\l},\, \, \, \l \in (0,\Lambda)\right\}$.\\
 Henceforth, given $\l \in(0,\Lambda)$, we use the notation $u_{\l}$ for the minimal solution. Let us define $a(x)=\l
qu_{\l}^{q-1}+pu_{\l}^{p-1}$ and let $\mu_1$ be the first eigenvalue of the following the problem
\begin{equation}\label{}
\left\{
\begin{array}{rcll}
(-\Delta)^s \phi- a(x)\phi & = & \mu_1 \phi& {\text{ in }}\O,\\
   \phi & > & 0 &{\text{ in }} \Omega, \\
   \mathcal{B}_{s}\phi & = & 0 &{\text{ in }} \RR^{N}\backslash \O.\\
\end{array}\right.
\end{equation}
Using closely the same argument as in the proof of Lemma 3.5 in \cite{ABC}, we can prove that
\begin{equation}\label{mu12}\mu_1 \geq 0.\end{equation}

It is clear that \eqref{mu12} is equivalent to
\begin{equation}\label{mu}
\| \phi\|^2 \geq \Int_{\O}a(x)\phi^2 dx
    \quad \forall \phi \in \H^s.
\end{equation}
Since $u_\l$ is a solution to ($P_\l$), testing the equation
against $u_\l$ itself, we find that
\begin{equation}\label{mu00}
\|u_\l\|^2=\l \|u_\l\|^{q+1}_{q+1}+\|u_\l\|^{p+1}_{p+1}.
\end{equation}
By (\ref{mu}), it follows that
\begin{equation}\label{mu000}
\|u_\l\|^2-\l q \|u_\l\|^{q+1}_{q+1}-p\|u_\l\|^{p+1}_{p+1}\geq 0.
 \end{equation}
By inserting these relations into \eqref{JFU}, we obtain
that $J_\l(u_\l)<0$, as desired.
\end{proof}
This gives point (1) in Theorem \ref{th1}. Thus, to complete the proof of Theorem \ref{th1},
we can now focus on the proof of point (3). To this end, we have the following result:
\begin{lemma}\label{P3}
 Problem ($P_\l$) has at least one solution if $\l=\Lambda$.
\end{lemma}

\begin{proof}
Let $\left\{\l_{n}\right\}$ be a sequence such that $\l_{n}\nearrow\Lambda$. We denote by $u_{n}\equiv u_{\l_n}$ the minimal solution to problem
($P_{\l_{n}}$), then the sequence $\{u_n\}_n$ is increasing in $n$.  Since $J_{\l_{n}}(u_{n})<0$, we get
\begin{eqnarray*}
  0 & > & J_{\l}(u_n)-\frac{1}{p+1}J'_{\l}(u_n) \\
    &\geq & (\frac{1}{2}-\frac{1}{p+1})\|u_n\|^2 +\l(\frac{1}{p+1}-\frac{1}{q+1})\|u_n\|^{q+1}_{q+1}\\
    &\geq & (\frac{1}{2}-\frac{1}{p+1})\|u_n\|^2 -\l(\frac{1}{q+1}-\frac{1}{p+1})\|u_n\|^{q+1}.
\end{eqnarray*}

Then, it follows that $\{u_{n}\}$ is bounded in $\H^{s}$.
Accordingly, we have that $u_n\rightharpoonup u^*$ in $\H^s$, for some $u^*\in \H^s$.
Since $\{u_n\}_n$ is increasing in $n$, using the fact that $(-\Delta)^su_n \geq 0$, recalling Lemma \ref{lmcp},
we conclude that $u_n\rightarrow u^*$ strongly in $\H^{s}$.
As a consequence, $u^*$ is a solution of ($P_{\lambda}$) for $\lambda=\Lambda$.
\end{proof}

\begin{remark} {\rm If $p \leq 2^*_s-1$ then using Theorem \ref{threg},
we can easily prove that $u^* \in L^{\infty}(\O)$, that means that $u^*$ is a
regular extremal solution.}
\end{remark}
 In view of Lemma \ref{P3}, we obtain point (3) of
Theorem \ref{th1}. The proof of Theorem \ref{th1} is thus complete.

\section{Proof of Theorem \ref{th2}}

  In this section we prove the existence of a second positive solution to ($P_\l$).

Since $p<\frac{N+2s}{N-2s}$, we observe that problem ($P_{\lambda}$) has a variational structure, indeed it is
the Euler-Lagrange equation of the energy functional in \eqref{JFU}. We note
that $J_{\l}$ is well defined, it is differentiable on $\H^s$ and for any $\varphi \in \H^s$,
$$( J_{\l}^{'}(u),\, \varphi) = \langle u,\,\varphi\rangle-\l\Int_{\O}|u|^{q}\varphi\,dx-\Int_{\O}|u|^{p}\varphi\,dx. $$
Thus critical points of the functional $J_{\l}$ are solutions to ($P_{\l}$).

To prove Theorem \ref{th2}, we will use a mountain pass-type argument. The proof goes as follows.
As in the local case, we can prove that the problem has a second
positive solution for $\l$ small. This follows using the mountain
pass theorem. For this purpose it is essential to have a first
solution which is a local minimum in $\H^{s}$. Let
$$ f_{\l}(r)= \left\{
\begin{array}{lll}
\l r^q+r^p,& {\mbox{if }}r\ge 0,\\ &\\ 0,&  {\mbox{if }}r<0,
\end{array}
\right.
$$
and
$$ F_{\l}(u)=\int_0^u f_{\l}(r)\,dr.$$
We define the functional $J_{\l}(u)=\frac{1}{2}\|u\|^2-\displaystyle\int_{\O}F_{\l}(u)$. Critical points of $J_{\l}$
correspond to solutions of ($P_{\l}$). Define the set
$$A=\{\lambda>0 \,:\, J_{\lambda}\hbox{ has a local minimum } u_{0,\lambda} \}. $$
 It is clear that if $\lambda \in A$ and $w_{\lambda}$ is a minimum of $J_{\lambda}$
 in $\H^{s}$, then $v=0$ is a local minimum of the functional
\begin{equation}\label{eq:funtras}
\hat{J}_\l(v)=\dfrac 12\|v\|^2-\int_\O G_{\l}(v)dx,
\end{equation}
where
$$
G_{\l}(v)=\int_0^v g_{\l}(r)\,dr $$ and $$ g_{\l}(r)= \left\{
\begin{array}{lll}
\l\left((u_{0,\l}(x)+r)^{q}-u_{0,\l}(x)^{q}\right)+(u_{0,\l}(x)+r)^{p}-u_{0,\l}(x)^{p},&
{\mbox{ if }}
r\ge 0,\\ &\\ 0,&
{\mbox{ if }}
r<0.&
\end{array}
\right.
$$
We can see  that $\hat{J}_\l$ possesses the mountain pass geometry. Thus, let
$v_{0} \in \H^{s}$ be such that $\hat{J}_\l(v_{0})<0$ and define
$$ \Gamma =\left\{ \gamma :\,[ 0,1] \rightarrow \H^{s}\;\,\gamma (0)=0,\,\gamma (1)=v_0\right\}\hbox{ and }\, c=\inf_{\gamma
\in \Gamma } \max_{ t\in[0,1]}\Phi_\l\left( \gamma (t)\right).
$$
We have that $c\geq 0$ and since $p < 2_{s}^{*}-1$, then $\hat{J}_\l$
satisfies the Palais-Smale condition. If $c>0$, then using the
Ambrosetti-Rabinowitz theorem we reach a non trivial critical point.
If $c=0$, then we use the Ghoussoub-Preiss Theorem, see \cite{gp}.

As a consequence if we start with a local minimum of the functional $\hat{J}_\l$,
then we obtain a second critical point of $\hat{J}_\l$, and hence a
second solution to ($P_\l$).

 Next, to show that problem ($P_{\lambda}$) has a second solution for all $\l \in (0,\Lambda)$,
 we follow some arguments similar to those developed by Alama
 in \cite{alm} taking into consideration the nonlocal nature of the operator.

We prove first, using a variational formulation of the Perron's method, that the functional
has a constrained minimum and then that this minimum is a local minimum in the whole $\H^s$.
To this end, we use a truncation technique and some energy estimates.

 Fix $\l_0\in (0,\Lambda)$ and let $\l_0 < \bar{\l} <\Lambda$. Define
$u_0,\,\bar{u}$ to be the minimal solutions to problem
($P_{\lambda}$) with $\l=\l_0$ and $\l=\bar{\l}$ respectively. By
definition we obtain that $u_0 < \bar{u}$. Let us define
$$ M=\{u\in
\H^s:\,\,0\le u\le \bar{u}\}.$$
It is clear that $u_0\in M$
and that $M$ is a convex closed subset of $\H^s$. Since
$J_{\lambda_0}$ is bounded from below in $M$ and lower semi-continuous,
then we get the existence of $\vartheta \in M$ such
that $$ J_{\lambda_0}(\vartheta)=\inf_{u\in M}J_{\lambda_0}(u).$$
Let $v$ be the unique solution to
$$
\left\{
\begin{array}{rcll}
(-\Delta)^s u & = & \l_0 u^q & {\text{ in }}\Omega,\\
   u & > &  0 &{\text{ in }} \Omega, \\
   \mathcal{B}_{s}u & = & 0 &{\text{ in }} \RR^{N}\backslash \Omega \,.\end{array}
\right.
$$
We have that $J_{\l_0}(v)<0$, and then $\vartheta \neq 0$. As in Theorem 2.4 in \cite{stw},
page 17, we conclude that $\vartheta$ is a solution to problem ($P_{\l}$).

If $\vartheta\neq u_{0}$, then the proof of Theorem \ref{th2}
is complete.
Accordingly, we can assume that $\vartheta =u_{0}$. We show that
\begin{equation}\label{TBP}
{\mbox{$\vartheta$ is a local minimum of $J_{\lambda_0}$.}}\end{equation} For this,
we argue by contradiction, and we
assume that $\vartheta$ is not a local
minimum of $J_{\lambda_0}$. Then there exists a sequence
$\{v_n\}\subset \H^s$ such that  $\|v_n-\vartheta\|_{\H^s}\to 0$ as $n\to \infty$ and
\begin{equation}\label{PTT}
J_{\lambda_0}(v_n)<J_{\lambda_0}(\vartheta).\end{equation}
We define $w_n=(v_n -\bar{u})_{+}$ and
$u_n=\max\{0,\min\{v_n,\bar{u}\}\}$. It is clear that $u_n\in M$ and
$$
u_n(x)= \left\{
\begin{array}{lll}
& 0 & \mbox{  if  }v_n(x)\le 0,\\ & v_n(x) & \mbox{  if  }0\le
v_n(x)\le \bar{u}(x),\\&  \bar{u}(x) & \mbox{  if  }\bar{u}(x)\le
v_n(x).
\end{array}
\right.
$$
Thus $u_n=v_n^+ - w_n$. Let $T_n= \{x\in \O:\,u_n(x)=v_n(x)\}$
and $S_n= \text{supp}\,\,\,w_n \cap \O$. Notice that $\text{supp}\,\,\,v_n^+ \cap \O=T_n\cup S_n$.
 We claim that
\begin{equation}\label{IsN876}
{\mbox{$|S_n|\to 0$ as $n\to \infty$.}}\end{equation}
To this end, let $\e>0$,
\begin{eqnarray*}&& E_n=\{x\in \O:\,v_n(x)\ge
\bar{u}(x)>\vartheta(x)+\d\}
\\ \mbox{  and   }&&
F_n=\{x\in \O:\,v_n(x)\ge\bar{u}(x)\mbox{ and }\bar{u}(x)\le \vartheta(x)+\d\},\end{eqnarray*} where $\d$ has
to be suitably chosen. Since
\begin{eqnarray*} 0 & = & |\{x\in \O:\,
\bar{u}(x)< \vartheta(x)\}|=\left|
\bigcap_{j=1}^\infty \left\{x\in \O:\,
\bar{u}(x)\le \vartheta(x)+\frac{1}{j}\right\}\right|\\ & = & \lim_{j\to
\infty}\left|\left\{x\in \O:\, \bar{u}(x)\le \vartheta(x)+\frac{1}{j}\right\}\right|,
\end{eqnarray*} then we get the existence of a suitable $\d_0=
\frac{1}{j_0}$ such that if $\d<\d_0$, then $$|\{x\in \O:\,
\bar{u}(x)\le \vartheta(x)+\d\}|\le \frac{\e}{2}.$$ Thus $|F_n|\le
\frac{\e}{2}$. Since $\|u_n-v_0\|_{L^2(\O)}\to 0$ as $n\to \infty$,
we get that for $\eta=\frac{\d^2\e}{2}$, if $n\ge n_0$, we have that
$$
\frac{\d^2\e}{2}\ge\int_\Omega|v_n-\vartheta|^2dx\ge \int_{E_n}
|v_n-\vartheta|^2dx \ge \d^2|E_n|. $$ Hence $|E_n|\le \frac{\e}{2}$.
Since $S_n\subset F_n\cup E_n$, we conclude that $|S_n|\le \e$ for
$n\le n_0$ and then the claim in~\eqref{IsN876} follows.

Now we define
$$H(u)=\dfrac{\l_{0}}{q+1} u_+^{q+1}+\dfrac{u_{+}^{p+1}}{p+1}.$$ Using the fact that
$$
\|v_n\|^2\ge \|v^+_n\|^2 + \|v^-_n\|^2,
$$
we obtain that
\begin{eqnarray*}
J_{\l_0}(v_n) & = & \frac {1}{2} \|v_n\|^2 -\io H(v_n)dx \\
&\ge & \frac{1}{2} \|v^+_n\|^2-\io H(v_n)dx+ \frac{1}{2} \|v^-_n\|^2\\
&=&\frac{1}{2}
\|v^+_n\|^2-\int_{T_n}H(u_n)dx-\int_{S_n}H(v_n)dx+\frac12
\|v^-_n\|^2 \\
&= & \frac {1}{2} \|v^+_n\|^2 -\int_{T_n}H(u_n)dx
-\int_{S_n}H(w_n+\bar{u})dx+
\frac{1}{2} \|v^-_n\|^2 \\
&=& J_{\l_0}(u_n) + \frac {1}{2}
\Big(\|v^+_n\|^2-\|u_n\|^2\Big)+\frac {1}{2} \|v^-_n\|^2
-\int_{S_n}\Big(H(w_n+\bar{u})-H(\bar{u})\Big)dx,
\end{eqnarray*}
where we have used the fact that
 $$\io
H(u_n)dx=\int_{T_n}H(u_n)dx+\int_{S_n}H(\bar{u})dx.
$$
Also, since $v^+_n=u_n+w_n$, then
$$
\frac{1}{2} \Big(\|v^+_n\|^2-\|u_n\|^2\Big)=\frac{1}{2}
\|w_n\|^2+\langle u_n, w_n\rangle
.$$
Using that
$$
\{w_n\neq 0\}=\{u_n=\bar{u}\},$$
we see that
$$
\langle u_n, w_n\rangle\ge \io (-\Delta)^s\bar{u}w_n dx\ge
\l\int_{S_n}\bar{u}^q w_n dx+\int_{S_n}\bar{u}^p w_n dx.
$$
Therefore, recalling
that $\bar{u}$ is a supersolution to problem ($P_{\lambda}$) for $\l=\l_0$, we conclude that
\begin{eqnarray*}
J_{\l_0}(v_n) &\ge&
J_{\l_0}(\vartheta)+\frac{1}{2}\|w_n\|^2+\frac{1}{2}\|v_n^-\|^2\\
&-&\int_{S_n}
\Big\{H(w_n+\bar{u})-H(\bar{u})-\l_{0}\bar{u}^{q}w_n-\bar{u}^{p}w_n\Big\}dx.
\end{eqnarray*}
Taking into account that
 $$ 0\le
\frac{1}{q+1}(w_n+\bar{u})^{q+1}-\frac{1}{q+1}\bar{u}^{q+1}-\bar{u}^{q}w_n\le
\frac{q}{2}\frac{w_n^2}{\bar{u}^{1-q}},$$ and using the Picone
inequality in Theorem \ref{PICONE}, we find that $$ \bar{\l}\io
\frac{w_n^2}{\bar{u}^{1-q}}dx \leq \io \frac{w_{n}^{2}}{\bar{u}}
(-\Delta)^s\bar{u} \le \|w_n\|^2.$$
Then, we obtain that
$$\l_{0}\io\bigg\{\frac{1}{q+1}(w_n+\bar{u})^{q+1}-\frac{1}{q+1}\bar{u}^{q+1}-\bar{u}^{q}w_n\bigg\}dx\le
\frac{q}{2}\displaystyle\io\frac{w_n^2}{\bar{u}^{1-q}}dx\leq
\frac{q}{2}\|w_n\|^2.$$
Moreover, since $2\leq p+1$,
$$ 0\le
\frac{1}{p+1}(w_n+\bar{u})^{p+1}-\frac{1}{p+1}\bar{u}^{p+1}-\bar{u}^{p}w_n\le
\frac{p}{2}w_n^2(w_n+\bar{u})^{p-1}\le
C(\bar{u}^{p-1}w_n^2+w_n^{p+1}).$$
Hence, using the Sobolev inequality and the fact that
$|S_n|\to 0$ as $n\to \infty$, we reach that
$$
\io\Big\{
\frac{1}{p+1}(w_n+\bar{u})^{p+1}-\frac{1}{p+1}\bar{u}^{p+1}-\bar{u}^{p}w_n\Big\}dx
\le o(1)\|w_n\|^2. $$ Hence
\begin{eqnarray*}
J_{\l_0}(v_n)&\ge&
J_{\l_0}(\vartheta)+\frac{1}{2}\|w_n\|^2(1-q-o(1))+\frac{1}{2}\|v_n^-\|^2\\
& \ge &
J_{\l_0}(\vartheta)+\frac{1}{2}\|w_n\|^2(1-q-o(1))+o(1).
\end{eqnarray*}
So we get that
$$0>J_{\l_0}(v_n)-J_{\l_0}(\vartheta)\geq \frac{1}{2}\|w_n\|^2(1-q-o(1))+\frac{1}{2}\|v_n^-\|^2.$$
Since $q<1$, we conclude that $w_n=v_n^-=0$ for $n$ large, so $v_n \in M$ and then
$$J_{\l_0}(v_n)\geq J_{\l_0}(\vartheta),$$
which is in contradiction with~\eqref{PTT}.

This completes the proof of~\eqref{TBP}.
{F}rom this, we have that~$\vartheta$ is a local minimum for $J_{\lambda_0}$, and
$\hat{J}_{\l_0}$ has $u=0$ as a local minimum and then $\hat{J}_{\l_0}$ has a nontrivial critical point $\hat{u}$.
As a consequence, $u = \vartheta +\hat{u}$ is a solution, different from $\vartheta$, of
problem ($P_\l$). This concludes the proof of Theorem \ref{th2}.
\begin{remark}
{\rm If we consider the odd symmetric version of problem ($P_{\l}$), namely,
\begin{equation}\label{odd prb}
\left\{
\begin{array}{rcll}
(-\Delta)^s u & = & \lambda |u|^{q-1}u+|u|^{p-1}u & {\text{ in }}\O,\\
     \\
\mathcal{B}_{s}u & = & 0 &{\text{ in }} \RR^{N}\backslash \O \,,
\end{array}\right.
\end{equation}
the associated functional
$$ I_\l(u)=\frac{1}{2}\|u\|^2-\frac{\l}{q+1}\|u\|^{q+1}_{q+1}-\frac{1}{p}\|u\|^{p+1}_{p+1} $$
is even. Then, for $p<\frac{N+2s}{N-2s}$, by using the Lusternik-Schnirelman min-max argument, it is possible to prove
that problem (\ref{odd prb}) has infinitely many solutions with negative energy, see \cite{ABC} and \cite{azor peral},
and following closely the arguments in \cite{amrab}, \cite{ABC} the same holds for solutions with positive energy.}
\end{remark}

\section*{Acknowledgements}
The authors would like to express their gratitude to the anonymous referee for his/her comments and suggestions that improve the last version of the manuscript.

\

Part of this work was carried out while the second author was visiting the {\it
Weierstra{\ss}-Institut f\"ur Angewandte Analysis und Stochastik} in Berlin.
He thanks the institute for the warm hospitality.

\end{document}